\documentclass[12pt]{amsart}

\usepackage{amssymb, amscd, txfonts}
\usepackage{graphicx}
\usepackage[all]{xy}

\numberwithin{equation}{section}

\newtheorem{theorem}{Theorem}[section]
\newtheorem{proposition}[theorem]{Proposition}
\newtheorem{lemma}[theorem]{Lemma}
\newtheorem{corollary}[theorem]{Corollary}
\newtheorem{problem}[theorem]{Problem}

\theoremstyle{definition}
\newtheorem{definition}[theorem]{Definition}
\newtheorem{example}[theorem]{Example}

\theoremstyle{remark}
\newtheorem{remark}[theorem]{Remark}

\newcommand{\Z}{\mathbb{Z}}
\newcommand{\Q}{\mathbb{Q}}

\newcommand{\C}{\mathbb{C}}

\newcommand{\proj}{{\mathbb P}}

\begin{document}

\title[]{On $K3$ surfaces which dominate Kummer surfaces}
\author[]{Shouhei Ma}
\address{Graduate~School~of~Mathematical~Sciences, the~University~of~Tokyo, 3-8-1~Komaba, Meguro-ku, Tokyo 153-8914, Japan}
\email{sma@ms.u-tokyo.ac.jp}
\maketitle

\begin{abstract} 
We study isogeny relations between $K3$ surfaces and Kummer surfaces. 
Specifically, we prove a Torelli-type theorem for the existence of rational maps from $K3$ surfaces to Kummer surfaces, 
and a Kummer sandwich theorem for $K3$ surfaces with Shioda-Inose structure. 
\end{abstract}

\section{Introduction}\label{sec1} 
In the present note we study rational maps between $K3$ surfaces in terms of their periods.   
Let $X$ be a complex algebraic $K3$ surface  
and $T_{X}$  be the transcendental lattice of $X$,  which is endowed with a natural Hodge structure.   
For a natural number $n>0$   
let $T_{X}(n)$ be the lattice obtained by multiplying the quadratic form on $T_{X}$ by $n$.   
In \cite{Ni2} Shafarevich posed the following question. 
   
\begin{problem}[\cite{Ni2} Question 1.1]\label{Shafarevich}
Let $X$ and $Y$ be complex algebraic $K3$ surfaces.  
Is it true that 
there exists a dominant rational map $X \dashrightarrow Y$ if and only if 
there exists a Hodge isometry $T_{X}\otimes{\Q}\simeq T_{Y}(n)\otimes{\Q}$ for some natural number $n$?
\end{problem}

Shafarevich's question is a variation of Torelli-type problem.  
It proposes to consider the ${\Q}$-Hodge structures $T_X\otimes{\Q}$ (up to scaling) for the existence of rational maps. 
A recent result of Chen \cite{Ch} shows that the answer is in general negative. 
On the other hand, the problem has been solved affirmatively in certain cases: 
for $K3$ surfaces $X$ with Picard number $\rho(X)=20$ (``singular $K3$ surfaces'') 
by Inose and Shioda \cite{S-I}, \cite{In} already before \cite{Ni2}; 
for $K3$ surfaces $X$ with $\rho(X)=19$ by Nikulin-Shafarevich \cite{Ni2}; 
Nikulin \cite{Ni2} studied rational maps obtained as compositions of double coverings.  
When both $X$ and $Y$ are Kummer surfaces, Problem \ref{Shafarevich} is obviously true 
by the corresponding property of Abelian surfaces. 
The first purpose of this note is to answer Problem \ref{Shafarevich} affirmatively 
when the target $Y$ is a Kummer surface. 
 
\begin{theorem}\label{main}
Let $X$ and $Y$ be complex algebraic $K3$ surfaces.   
Assume that $Y$ is dominated by some Kummer surface, 
e.g., $Y$ admits a Shioda-Inose structure or $Y$ itself is a Kummer surface.   
Then there exists a dominant rational map $X \dashrightarrow Y$    
if and only if  there exists a Hodge isometry $T_{X}\otimes {\Q} \simeq T_{Y}(n)\otimes{\Q}$ 
for some natural number $n$.   
\end{theorem}

In order to produce a desired map $X\dashrightarrow Y$, 
we will compose the following three types of rational maps: 
$(1)$ double coverings,   
$(2)$ rational maps between Kummer surfaces induced by isogenies of Abelian surfaces, and   
$(3)$ multiplication maps from elliptic $K3$ surfaces to the associated Jacobian fibrations. 
The first two have been also used in \cite{S-I}, \cite{In}, and \cite{Ni2}.  
A new ingredient of this paper is a systematic use of the third type of rational maps. 
In the course of the proof, 
we shall characterize those $K3$ surfaces which dominate Kummer surfaces by their Hodge structures.   

The approach of Inose and Shioda for Problem \ref{Shafarevich}  
was to use a Kummer sandwich theorem, 
which roughly says that a singular $K3$ surface is two-isogenous to a Kummer surface. 
Recently the Kummer sandwich theorem has been extended to a larger class of $K3$ surfaces by Shioda \cite{Sh} and has found some arithmetic applications.   
The $K3$ surfaces studied in \cite{Sh} are characterized by 
the existence of Shioda-Inose correspondences with products of elliptic curves.   
The second purpose of this note is to prove a Kummer sandwich theorem 
for all complex algebraic $K3$ surfaces with Shioda-Inose structure (Theorem \ref{sandwich for Shioda-Inose}). 
It is independent of Theorem \ref{main}, 
and shows a more precise isogeny relation between 
Kummer surfaces and $K3$ surfaces with Shioda-Inose structure.  

Throughout this paper, the varieties are assumed to be complex algebraic. 
The transcendental lattice of an algebraic surface $X$ will be denoted by $T_X$.  
By $U$ we denote the rank $2$ even indefinite unimodular lattice.  
By $E_8$ we denote the rank $8$ even {\rm negative-definite\/} unimodular lattice.  
For a lattice $L=(L, (,)_{L})$ and a natural number $n$,  
we denote by $L(n)$ the scaled lattice $(L, n(,)_{L})$.

\noindent\textbf{Acknowledgements.}   
The author wishes to express his gratitude to 
Professor Ken-Ichi Yoshikawa for his advice and encouragement.  
He also thanks A. Mehran and M. Sch\"utt for their useful comments.  
He is grateful to the referee for reading the manuscript carefully and for pointing out the reference \cite{Ch}. 
This work was supported by Grant-in-Aid for JSPS fellows [21-978].

\section{Kummer sandwich theorem}\label{sec2}

Let $X$ be an algebraic $K3$ surface. 
Recall that a {\it Nikulin involution\/} of $X$ is 
an involution $\iota : X \to X$ which acts trivially on $H^{2, 0}(X)$.  
A Nikulin involution of $X$ canonically corresponds to a double covering $X \dashrightarrow Y$ 
to another $K3$ surface $Y$.  
Indeed, if we have a double covering $\pi : X \dashrightarrow Y$, 
then the covering transformation of $\pi $ is a Nikulin involution of $X$. 
Conversely, for a Nikulin involution $\iota $ of $X$ 
the minimal resolution $Y = \widetilde{X/\langle \iota \rangle }$ of the quotient surface 
is a $K3$ surface (\cite{Ni1}), 
and we have the rational quotient map $\pi : X \dashrightarrow Y$ of degree $2$.  
The transcendental lattices $T_{X}$ and $T_{Y}$ are related by the chain of inclusions 
\begin{equation}\label{relation of the Hodge structures}
2T_{Y} \subseteq \pi _{\ast}T_{X} = T_{X}(2) \subseteq T_{Y}, 
\end{equation}
which preserves the quadratic forms and the Hodge structures.  

Nikulin \cite{Ni1}, \cite{Ni2} and Morrison \cite{Mo} developed the lattice-theoretic aspect of Nikulin involution.  
Let us denote  
\begin{eqnarray}
\Lambda _{0} &  :=  &                  E_{8}(2)\oplus U^{3},         \label{the important invariant lattice 1}\\   
\Lambda _{1} &  :=  & \frac{1}{2}E_{8}(2)\oplus U^{3}.         \label{the important invariant lattice 2}
\end{eqnarray}
We regard $\Lambda _{0}$ as a submodule of $\Lambda _{1}$ in a natural way.   
Then $\Lambda _{1}$ is the dual lattice of $\Lambda _{0}$.   
The following proposition reduces the construction of a Nikulin involution 
to a purely arithmetic problem. 

\begin{proposition}[\cite{Ni2} Section 2.1 and Lemma 2.2.4]\label{reduction to arithmetic problem} 
Let $X$ be an algebraic $K3$ surface. 
Suppose that one is given a primitive embedding $T_{X} \subset \Lambda _{0}$ of lattices.  
Then there exists a Nikulin involution $\iota : X \to X$ such that,   
if we denote $Y = \widetilde{X/\langle \iota \rangle }$,  
then $T_{Y}$ is Hodge isometric to the lattice 
\begin{equation}\label{Hodge structure of the quotient}
T := \left( \, T_{X} \otimes {\Q} \, \cap \, \Lambda _{1}  \right) \, (2),      
\end{equation}
where the Hodge structure of $T$ is induced from $T_{X}$.  
Conversely, if one has a rational map $X \dashrightarrow Y$ of degree $2$ 
to a $K3$ surface $Y$, then there exists a primitive embedding $T_{X} \subset \Lambda _{0}$  
such that $T_{Y}$ is Hodge isometric to the lattice $T$ defined by $(\ref{Hodge structure of the quotient})$.         
\end{proposition}

Shioda-Inose structure is a special kind of Nikulin involution.  

\begin{definition}[\cite{Mo}]\label{def of S-I structure}
An algebraic $K3$ surface $X$ admits a {\it Shioda-Inose structure\/} 
if there exists a Kummer surface $Y={\rm Km\/}A$ and a rational map $\pi : X \dashrightarrow Y$ of degree $2$ 
such that $\pi _{\ast}$ induces a Hodge isometry $T_{X}(2) \simeq T_{Y}$.  
\end{definition} 

There is a lattice-theoretic characterization of  
$K3$ surfaces admitting Shioda-Inose structures due to Morrison.   

\begin{theorem}[\cite{Mo} Theorem 6.3]\label{Morrison condition} 
An algebraic $K3$ surface $X$ admits a Shioda-Inose structure 
if and only if there exists a primitive embedding 
$T_{X} \hookrightarrow U^{3}$ of lattices.   
\end{theorem}

Shioda \cite{Sh}, extending the work of Inose \cite{In}, 
proved a Kummer sandwich theorem for elliptic $K3$ surfaces with section and with two $II^{\ast}$-fibers 
over an arbitrary algebraically closed field of characteristic $\ne 2, 3$.  
When the ground field is ${\C}$, one can characterize the $K3$ surfaces studied in \cite{Sh} by the existence of 
Shioda-Inose structures such that the corresponding Abelian surfaces are products of elliptic curves. 
Here we shall derive in a transcendental way 
a Kummer sandwich theorem for all complex algebraic $K3$ surfaces with Shioda-Inose structure.   
We denote the Dynkin diagram of $E_{8}$ by 
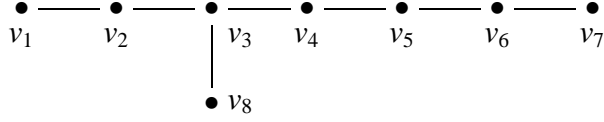
\begin{equation*}\label{Dynkin diagram}
\xygraph{
    \bullet ([]!{+(0,-.3)} {v_1}) - [r]
    \bullet ([]!{+(0,-.3)} {v_2}) - [r]
    \bullet ([]!{+(.3,-.3)} {v_3}) (
        - [d] \bullet ([]!{+(.3,0)} {v_8}),
        - [r] \bullet ([]!{+(0,-.3)} {v_4})
        - [r] \bullet ([]!{+(0,-.3)} {v_5})
        - [r] \bullet ([]!{+(0,-.3)} {v_6})
        - [r] \bullet ([]!{+(0,-.3)} {v_7})
)}
\end{equation*}
We identify the ${\Z}$-modules underlying $E_{8}$ and $E_{8}(2)$ in a natural way, 
and regard the above set $\{ v_{i} \} _{i=1}^{8}$ as a basis of $E_{8}(2)$.  
Then we have 
$(v_{i}, v_{i})=-4$ for $i=1,\cdots ,8$,  
$(v_{i}, v_{j})=2$ if  $v_{i}$ and $v_{j}$ are connected by an edge, 
and $(v_{i}, v_{j})=0$ otherwise. 
Let $\{ e_{i}, f_{i} \} _{i=1}^{3}$ be the standard basis of $U^{3}$.   
We have 
$(e_{i}, e_{j}) = (f_{i}, f_{j}) =0$ 
and $(e_{i}, f_{j})=\delta _{i j}$. 
For $i=1, 2, 3$, we define the vectors $l_{i}, m_{i} \in \Lambda _{0} = E_{8}(2) \oplus U^{3}$ by 
\begin{eqnarray*}
l_{1}    & = &  -v_{5}+v_{7}+2(e_{1}+f_{1}),          \\ 
m_{1}  & = &  -v_{4},                                                 \\ 
l_{2}    & = &    v_{1}+v_{8}+2(e_{2}+f_{2}),         \\ 
m_{2}  & = &   v_{2},                                                 \\ 
l_{3}     & = &   v_{7}+v_{8}+2(e_{1}+e_{2}+e_{3}+f_{3}),      \\  
m_{3}  & = &   v_{6},  
\end{eqnarray*}
and put $L := \langle l_{1}, m_{1}, l_{2}, m_{2}, l_{3}, m_{3}  \rangle $.  

\begin{lemma}\label{embedding of U(2)^{3}}
The sublattice $L\subset \Lambda _{0}$ has the following properties.  

$(1)$ $L\simeq U(2)^{3}$.  

$(2)$ $L\subset 2\Lambda _{1}$.   

$(3)$ $L$ is a primitive sublattice of $\Lambda _{0}$.   
\end{lemma}

\begin{proof}
We can extend the set $\{ l_{i}, m_{i} \} _{i=1}^{3}$ to a ${\Z}$-basis of $\Lambda _{0}$  
by adding the set of vectors $\{ v_{3}, v_{5}, e_{1}, f_{1}, e_{2}, f_{2}, e_{3}, f_{3} \} $.  
Thus $L$ is primitive in $\Lambda _{0}$.  
The assertion $(2)$ is obvious 
and the assertion $(1)$ is proved by direct calculations.  
\end{proof} 

\begin{theorem}\label{sandwich for Shioda-Inose} 
Let $X$ be an algebraic $K3$ surface admitting a Shioda-Inose structure 
$X \dashrightarrow Y = {\rm Km\/}A$.  
Then there exists a Nikulin involution $\iota $ on $Y$ such that 
the minimal resolution of the quotient surface $Y/\langle \iota \rangle $ is isomorphic to $X$.  
In particular, one has the following sequence of rational maps of degree 2:  
\begin{equation}\label{sandwich diagram}
{\rm Km\/}A \dashrightarrow X \dashrightarrow {\rm Km\/}A.  
\end{equation}
\end{theorem}

\begin{proof}
By the definitions,   
we have the Hodge isometries 
\begin{equation*}
T_{Y} \simeq T_{A}(2),  \; \; \; T_{X}\simeq T_{A}.  
\end{equation*}
Since $T_{A}$ is embedded into $H^{2}(A, {\Z}) \simeq U^{3}$ primitively,  
there exists a primitive embedding 
$\varphi : T_{Y} \hookrightarrow U(2)^{3}$.   
By composing $\varphi $ with an isometry $U(2)^{3} \simeq L$,   
we obtain a primitive embedding 
$\psi : T_{Y} \hookrightarrow \Lambda _{0}$ such that 
$\psi ( T_{Y} ) \subset 2\Lambda _{1}$.  
We have 
\begin{equation*}
\psi (T_{Y}) \otimes {\Q} \: \cap \: \Lambda _{1}  \; = \; 
\frac{1}{2}\psi (T_{Y}).  
\end{equation*}
By Proposition \ref{reduction to arithmetic problem},  
there exists a Nikulin involution $\iota : Y \to Y$ 
such that for the minimal resolution $Z$ of $Y/\langle \iota \rangle $ 
the transcendental lattice $T_{Z}$ is Hodge isometric to 
\begin{equation*}
\frac{1}{2}T_{Y}(2) \simeq \frac{1}{2}T_{A}(4) \simeq T_{A} \simeq T_{X}.  
\end{equation*}
Since a Hodge isometry $T_{Z} \simeq T_{X}$ can be extended to 
a Hodge isometry $H^{2}(Z, {\Z}) \simeq H^{2}(X, {\Z})$ (cf. \cite{Mo} Corollary 2.10),  
we have $Z \simeq X$ by the Torelli theorem.  
\end{proof} 

The rational quotient map  
$\pi:{\rm Km\/}A \dashrightarrow X$   
constructed in Theorem \ref{sandwich for Shioda-Inose} 
induces a Hodge isometry    
$\pi ^{\ast} : T_{X}(2) \to T_{ {\rm Km\/}A}$.   
Thus a $K3$ surface $X$ with Shioda-Inose structure can be defined 
not only as a double cover of a Kummer surface ${\rm Km\/}A$ but also as a double quotient of ${\rm Km\/}A$, 
which exhibits an isogeny relation between $X$ and ${\rm Km\/}A$.  
Unfortunately, as we rely on the Torelli theorem, 
our Kummer sandwich theorem is not explicit as in \cite{In}, \cite{Sh}, 
and our argument works only over ${\C}$.   

$K3$ surfaces with Shioda-Inose structure are particular double covers of Kummer surfaces. 
Now, is it true in general that 
a double cover $X$ of a Kummer surface ${\rm Km\/}A$ admits a double covering 
${\rm Km\/}A \dashrightarrow X$ of the opposite direction, 
as like isogenies of elliptic curves?  
Here is a negative example. 




\begin{example}\label{nega exple}
Let $A$ be an Abelian surface with $T_{A} \simeq U \oplus \langle 2 \rangle \oplus  \langle -2 \rangle $ 
and $X$ be the $K3$ surface with $T_{X}$ Hodge isometric to $2T_{A}$.                                 
Then there exists a rational map $X \dashrightarrow {\rm Km\/}A$ of degree $2$,       
but  there does {\it not\/} exist a rational map ${\rm Km\/}A \dashrightarrow X$ of degree $2$. 
\end{example}

\begin{proof}  
The existence of a double covering $X \dashrightarrow {\rm Km\/}A$ follows from 
Mehran's criterion for double covers of Kummer surfaces (\cite{Me} Theorem 3.1). 
Suppose that we have a rational map ${\rm Km\/}A \dashrightarrow X$ of degree $2$.  
By Proposition \ref{reduction to arithmetic problem}  
there exists a primitive embedding $T_{A}(2) \hookrightarrow \Lambda _{0}$ such that  
\begin{equation}\label{a condition}
T_{A}(2)\otimes {\Q} \: \cap \: \Lambda _{1} = T_{A}(2). 
\end{equation}
Via this embedding, we regard $T_A(2)$ as a primitive sublattice of $\Lambda_0$. 
Let $\pi : T_{A}(2) \to U^{3}$ be the orthogonal projection,  
which is injective by the condition $(\ref{a condition})$.  
Let $M$ be the lattice $\pi (T_{A}(2))$ and $N$ be the primitive closure of $M$ in $U^{3}$.   
By the condition $(\ref{a condition})$ again,  
the Abelian group $N/M$ has no $2$-component.   
For an even lattice $L$ 
let $L^{\vee}$ be the dual lattice of $L$,  
$D_{L}=L^{\vee}/L$ be the discriminant group of $L$, 
and $(D_{L})_{2}$ be the $2$-component of $D_{L}$.   
We see from the inclusions 
$M \subset N\subset N^{\vee} \subset M^{\vee}$  
that \begin{equation}\label{isomorphisms of 2-components}
(D_{M})_{2} \simeq (D_{N})_{2} \simeq (D_{N^{\perp}\cap U^{3}})_{2}.  
\end{equation}
The second isomorphism follows from the fact that $N$ is a primitive sublattice of the unimodular lattice $U^{3}$.   
In particular,  
the length of $(D_{M})_{2}$ is less than or equal to $2$.   
On the other hand,  
we have $(v, w) \in 2{\Z}$ for every $v, w \in M$.   
Thus we have $\frac{1}{2}M \subset M^{\vee}$,   
which is absurd.   
\end{proof}

\begin{remark}
It follows from \cite{Ni2} Theorem 1.3 that  
for $X$ and $A$ as in Example \ref{nega exple}, 
there nevertheless exists a rational map  
${\rm Km\/}A \dashrightarrow X$ of degree $2^{\mu}$  
for some $\mu > 1$.    
\end{remark}

\section{Rational maps to Kummer surfaces}\label{sec3} 

In this section we study rational maps from $K3$ surfaces to Kummer surfaces in general.  
We shall use the following. 

\begin{proposition}[\cite{Ma} Section 4]\label{relative Jacobian}
Let $X$ and $Y$ be algebraic $K3$ surfaces with ${\rm rk}(T_X)={\rm rk}(T_Y)\leq9$ such that 
there exists an embedding $T_X \to T_Y$ of lattices preserving the periods. 
Then there exists a sequence $X_1=X, X_2, \cdots, X_n=Y$ of $K3$ surfaces such that 
$X_{i+1}$ is isomorphic to the surface underlying the Jacobian fibration of an elliptic fibration $\pi_i : X_i \to {\proj}^1$. 
In particular, for a line bundle $L \in {\rm Pic}(X)$ we have a rational map $X_{i} \dashrightarrow X_{i+1}$ defined by 
$x \mapsto \mathcal{O}_{F}(dx) \otimes L^{-1}$ where $F$ is the $\pi_{i}$-fiber containing $x \in X_i$ and $d=(L.F)$. 
\end{proposition}

We shall characterize $K3$ surfaces $X$ dominating Kummer surfaces by the lattices $T_X$.  

\begin{proposition}\label{generalized Morrison criterion}
For an algebraic $K3$ surface $X$ the following conditions are equivalent.   

$({\rm i\/})$ There exists a dominant rational map 
$X \dashrightarrow {\rm Km\/}A$ to some Kummer surface ${\rm Km\/}A$.  

$({\rm ii\/})$ There exists an embedding $T_{X}\otimes{\Q}\hookrightarrow U^{3}\otimes {\Q}$  
of quadratic spaces.  

$({\rm iii\/})$ There exists an embedding $T_{X}\hookrightarrow U^{3}$ of lattices.  
\end{proposition}

\begin{proof}
$({\rm i\/}) \Rightarrow ({\rm ii\/})$:  
A rational map 
$f : X \dashrightarrow {\rm Km\/}A$ 
of finite degree $d$ induces a Hodge isometry 
\begin{equation*}
f_{\ast} : T_{X}(d) \otimes {\Q} \stackrel{\simeq}{\to} T_{{\rm Km\/}A} \otimes {\Q} \simeq T_{A}(2) \otimes {\Q}. 
\end{equation*}
Then the quadratic space $T_{X}\otimes {\Q}$ is isometric to $T_{A}(2d)\otimes {\Q}$ and thus is embedded into   
$H^{2}(A, {\Q})(2d) \simeq U^{3}(2d) \otimes {\Q}$. 
By the property $U^{3}(2d)\otimes {\Q} \simeq U^{3} \otimes {\Q}$ of the lattice $U$, 
we obtain an embedding $T_{X} \otimes {\Q} \hookrightarrow U^{3} \otimes {\Q}$ of quadratic spaces. 

$({\rm ii\/}) \Rightarrow ({\rm iii\/})$:  
Recall that an even lattice of rank $r$ can be embedded (primitively) into $U^r$. 
In particular, we may assume that ${\rm rk\/}(T_{X}) = 4$ or $5$.   
When ${\rm rk\/}(T_X) = 4$, the condition $({\rm ii\/})$ is equivalent to 
the existence of an embedding $U\otimes{\Q}\to T_X\otimes{\Q}$ 
by Witt's theorem for $(T_X\otimes{\Q})^{\perp}\cap U^3\otimes{\Q}$. 
Thus we have an isotropic vector in $T_X$.  
Let $T$ be a maximal even overlattice of $T_X$. 
A primitive isotropic vector $v\in T$ induces an embedding $U\to T$ because $(v, T)={\Z}$.    
Hence $T \simeq U \oplus L$ for some rank $2$ lattice $L$ so that $T$ can be embedded into $U^3$.   
When ${\rm rk\/}(T_{X}) = 5$, as in the case of ${\rm rk\/}(T_X) = 4$,  
the condition $({\rm ii\/})$ is equivalent to the existence of a rank $2$ totally isotropic sublattice of $T_X$.  
Then every maximal even overlattice of $T_X$ is of the form $T=U^2\oplus L$, ${\rm rk\/}(L)=1$,  
and thus can be embedded into $U^3$.    

$({\rm iii\/}) \Rightarrow ({\rm i\/})$: 
We fix an embedding $T_{X} \subset U^{3}$.  
Let $T$ be the primitive closure of $T_{X}$ in $U^{3}$ and endow $T$ with the Hodge structure induced from $T_{X}$.  
We regard $T$ as a primitive sublattice of $U^{3}\oplus E_{8}^{2}$.  
By the surjectivity of the period map,  
there exists a $K3$ surface $Y$  
with $T_{Y}$ Hodge isometric to $T$.   
We have an embedding  
$T_{X} \hookrightarrow T_{Y}$ of finite index 
which preserves the periods.     
It follows from Proposition \ref{relative Jacobian} that 
there exists a dominant rational map $X \dashrightarrow Y$. 
Since the lattice $T_{Y}$ can be embedded primitively into $U^{3}$,  
the $K3$ surface $Y$ admits a Shioda-Inose structure $Y \dashrightarrow  {\rm Km\/}A$ by Theorem \ref{Morrison condition}.  
\end{proof}  

Proposition \ref{generalized Morrison criterion} is analogous to Theorem \ref{Morrison condition}:  
replacing Shioda-Inose structures by general rational maps corresponds to 
replacing primitive embeddings of lattices by embeddings of rational quadratic spaces. 

An Abelian surface $A$ is a product of two elliptic curves if and only if  
$T_{A}$ can be embedded primitively into $U^{2}$.  
Hence by a similar argument as in the above proof we have 
the following variant of Proposition \ref{generalized Morrison criterion}.   

\begin{proposition}\label{a variant of generalized Morrison criterion}
For an algebraic $K3$ surface $X$    
the following conditions are equivalent.  

$({\rm i\/})$ There exists a dominant rational map 
$X \dashrightarrow {\rm Km\/}A$ 
to some Kummer surface ${\rm Km\/}A$,  
where $A$ is a product of two elliptic curves.   

$({\rm ii\/})$ There exists an embedding 
$T_{X} \otimes {\Q} \hookrightarrow U^{2} \otimes {\Q}$  
of quadratic spaces.  

$({\rm iii\/})$ There exists an embedding $T_{X}\hookrightarrow U^{2}$ of lattices.  
\end{proposition}

By using Proposition \ref{generalized Morrison criterion} 
we deduce the next theorem, from which Theorem \ref{main} follows immediately.  

\begin{theorem}\label{rational map to Kummer}
Let $X$ be an algebraic $K3$ surface 
and ${\rm Km\/}A$ be an algebraic Kummer surface.  
Then there exists a dominant rational map $X \dashrightarrow  {\rm Km\/}A$ 
if and only if there exists a Hodge isometry 
$T_{X} \otimes {\Q} \simeq T_{A}(n)  \otimes {\Q}$ 
for some natural number $n$.  
\end{theorem} 
  
\begin{proof}
It suffices to prove the ``if'' part.  
Assume the existence of a Hodge isometry $T_{X} \otimes {\Q} \simeq T_{A}(n)  \otimes {\Q}$.  
As $T_{A}\otimes {\Q}$ is embedded into $U^{3}\otimes {\Q}$,  
by Proposition \ref{generalized Morrison criterion}  
we can find a Kummer surface  ${\rm Km\/}B$ 
and a finite rational map $X \dashrightarrow  {\rm Km\/}B$.  
Since we have a Hodge isometry 
$T_{B}(m)\otimes {\Q} \simeq T_{A}\otimes {\Q}$ 
for some natural number $m$, 
the Abelian surface $B$ is isogenous  to the Abelian surface $A$.  
Thus there exists a dominant rational map 
${\rm Km\/}B \dashrightarrow  {\rm Km\/}A$. 
\end{proof}

\begin{corollary}\label{reflect rational map}
Let $X$ be an algebraic $K3$ surface and 
$A$ be an Abelian surface.  
If we have a dominant rational map $A \dashrightarrow X$,  
then there exists a dominant rational map $X \dashrightarrow {\rm Km\/}A$.  
\end{corollary}


\begin{corollary}\label{isogeny via rat map}
Let $X$ and $Y$ be algebraic $K3$ surfaces dominated by some Kummer surfaces.  
Then there exists a dominant rational map $X \dashrightarrow Y$ 
if and only if 
there exists a dominant rational map $Y \dashrightarrow X$. 
\end{corollary} 

Thus, as like Inose's paper \cite{In}, 
we are able to define a notion of isogeny for those $K3$ surfaces dominated by Kummer surfaces  
by the existence of rational map.


\begin{thebibliography}{99}

\bibitem{Ch} Chen, X. 
{\it Self rational maps of $K3$ surfaces.\/} 
preprint, arXiv: 1008.1619. 

\bibitem{In}  Inose, H.  
{\it Defining equations of singular $K3$ surfaces and a notion of isogeny.\/} 
Proceedings of the International Symposium on Algebraic Geometry, pp. 495--502, Kinokuniya Book Store, 1978.

\bibitem{Ma} Ma, S. 
{\it On the 0-dimensional cusps of the K\"ahler moduli of a $K3$ surface.\/}   
Math. Ann. \textbf{348} (2010), no.1, 57--80.

\bibitem{Me} Mehran, A. 
{\it Double cover of Kummer surfaces.\/} 
Manuscripta Math. \textbf{123} (2007), 205--235. 

\bibitem{Mo} Morrison, D. R. 
{\it On $K3$ surfaces with large Picard number.\/}  
Invent. Math. \textbf{75} (1984), no. 1, 105--121.

\bibitem{Ni1}  Nikulin, V. V. 
{\it Finite groups of automorphisms of Kahlerian $K3$ surfaces.\/} 
Trudy Moskov. Mat. Obshch. \textbf{38} (1979), 75--137.

\bibitem{Ni2} Nikulin, V. V. 
{\it On rational maps between $K3$ surfaces.\/}  
Constantin Caratheodory: an international tribute, 964--995, World Sci. Publ., 1991.


\bibitem{Sh} Shioda, T.
{\it Kummer sandwich theorem of certain elliptic $K3$ surfaces.\/}   
Proc. Japan Acad. Ser. A.  \textbf{82} (2006), no. 8, 137--140. 

\bibitem{S-I} Shioda, T.; Inose, H.
{\it On singular $K3$ surfaces.\/}  
Complex analysis and algebraic geometry, pp. 119--136. Iwanami Shoten,  1977. 

\end{thebibliography}
\end{document}